\documentclass[10pt]{amsart}
\usepackage[titletoc]{appendix}
\usepackage{amssymb}
\usepackage{amsfonts}
\usepackage{amsthm}
\usepackage{amsmath,amscd}
\usepackage[margin=1.5in]{geometry} 
\usepackage[
backend=biber,
style=alphabetic,
sorting=nyt
]{biblatex}
\usepackage[utf8]{inputenc}
\usepackage[english]{babel}
 
\usepackage{enumerate}
\usepackage{verbatim}
\usepackage{mathrsfs}
\usepackage{url}

\makeatletter
\def\@tocline#1#2#3#4#5#6#7{\relax
  \ifnum #1>\c@tocdepth 
  \else
    \par \addpenalty\@secpenalty\addvspace{#2}%
    \begingroup \hyphenpenalty\@M
    \@ifempty{#4}{%
      \@tempdima\csname r@tocindent\number#1\endcsname\relax
    }{%
      \@tempdima#4\relax
    }%
    \parindent\z@ \leftskip#3\relax \advance\leftskip\@tempdima\relax
    \rightskip\@pnumwidth plus4em \parfillskip-\@pnumwidth
    #5\leavevmode\hskip-\@tempdima
      \ifcase #1
       \or\or \hskip 1em \or \hskip 2em \else \hskip 3em \fi%
      #6\nobreak\relax
    \hfill\hbox to\@pnumwidth{\@tocpagenum{#7}}\par
    \nobreak
    \endgroup
  \fi}
\makeatother

\theoremstyle{plain}
   \newtheorem{theorem}{Theorem}

\theoremstyle{definition}
    \newtheorem{definition}{Definition}

    \newtheorem{lemma}[theorem]{Lemma}
\newtheorem{example}{Example}

\theoremstyle{remark}
\newtheorem{remark}{Remark}

\newcommand{\NN}{\mathbb{N}}

\newcommand{\RR}{\mathbb{R}}

\makeatletter

\begin{document}

\title{An introduction to geodesics: the shortest distance between two points}
\author{Andrew Tawfeek}
\address{Bristol Community College, 777 Elsbree St, Fall River, MA}
\email{atawfeek11@bristolcc.edu}
\maketitle

\begin{abstract}
We give an accessible introduction and elaboration on the methods used in obtaining a geodesic, which is the curve of shortest length connecting two points lying on the surface of a function. This is found through computing what's known as the variation of a functional, a ``function of functions'' of sorts. Geodesics are of great importance with wide applications, e.g. dictating the path followed by aircraft (great-circles), how light travels through space, assist in the process of mapping a 2D image to a 3D surface, and robot motion planning.
\end{abstract}

\tableofcontents

\section{Introduction}
Let $\mathcal{S}$ be a surface defined by a vector equation $\vec{r} (u,v)$. A \textit{geodesic} is the shortest curve lying on $\mathcal{S}$ connecting two points on the surface of $\mathcal{S}$. This is done by minimizing the functional 
$$J[u,v]=\int_{t_0}^{t_1} \sqrt{Eu'^2+2Fu'v'+Gv'^2} \ dt,$$
where $E$, $F$, and $G$ are the coefficients of the \textit{first fundamental (quadratic) form of the surface}.

In this paper we provide the necessary background information needed to understand the concept of a geodesic, rigorously proving each step along the way. In Section 4, multiple Lemmas are put forth concerning function spaces which later greatly assist us in Section 5 in proving that a necessary condition for a function to be an extremum for a functional is for it to satisfy the Euler-Lagrange Differential Equation. Section 6 dives heavily into the concept of a geodesic and provides two different methods of obtaining it; followed by a few examples of both methods in use.

\

\textbf{Acknowledgments.} I wish to thank the Bristol Community College Honors Program for providing this unique opportunity to spend my Fall 2016 semester completing this honors component for my Calculus III course. I am indebted to Zachary Wolfson, my advisor for this project, for his guidance and support at every step, as well as shaping the way I think of mathematics today.

\section{Function Spaces}
Analogous to the study of $n$-variable functions, where $(x_1 \ ,x_2 \ , \dots ,x_n) \in \mathbb{R}^n$, function spaces employ similar concepts. Each function $y(x)$ is regarded to belong to a certain function space containing other functions with similar properties. Spaces whose elements are functions are called \textit{function spaces}.

\

In order to properly discuss function spaces, we must first introduce the concepts of a \textit{norm} and a \textit{normed linear space}.

\begin{definition}
A \textit{linear space} is a set $\Re$ of elements $x$ of any kind (numbers, vectors, matrices, functions, etc.), for which the operations of addition and multiplication by real numbers $a, \ b, \ \dots$ are defined and adhere to the following axioms:
\begin{enumerate} 
\item   $x_i + x_j = x_j + x_i$;
\item $(x_i + x_j) + x_k = x_i+ (x_j+x_k)$;
\item  There exits an element 0 such that $x_i + 0 = x_i$;
\item There exists an element $-x_i$ such that $x_i+(-x_i)=0$;
\item $1 \cdot x_y = x_i$;
\item $a(bx_i)=(ab)x_i$;
\item $(a+b)x_i = ax_i + bx_i$;
\item $a(x_i + x_j) = ax_i + ax_j$.
\end{enumerate}
\end{definition}
\begin{definition}
A linear space $\Re$ is said to be \textit{normed}\footnote[1]{Observe, for example, that for a linear space of vectors, the norm is their magnitude, and a for a linear space of numbers, it would just be their absolute value.} if each element $x \in \Re$ is assigned a non-negative number $||x||$, called the \textit{norm} of $x$, such that:
\begin{enumerate}
\item $||x|| = 0$ if and only if $x=0$;
\item $||ax|| = |a| \cdot ||x||$;
\item $||x+y|| \leq ||x|| + ||y||$.
\end{enumerate}
\end{definition}

Notice that by defining $||x-y||$ as the distance between $x$ and $y$, the norm induces a metric, that is, it satisfies the conditions:
\begin{enumerate}
\item $||x+y|| = 0$ if, and only if, $x=y$;
\item $||x+y|| = ||x+y||$;
\item $||x+y|| + ||y+z|| \geq ||x+z||$.
\end{enumerate}

The following normed linear spaces are important for definitions and theorems to come:
\begin{enumerate}
\item The space $C$, or more precisely $C(a,b)$, consisting of all continuous functions $y(x)$ defined on an interval $a \leq x \leq b$. By addition and multiplication of elements of $C$, we mean ordinary addition of functions and multiplication of functions by numbers. The norm of this space is defined as the maximum of the absolute value, i.e.
$$||y||_0 = \underset{a \leq x \leq b}{max} \ |y(x)|.$$

Which satisfies the propeties of a norm due to the properties of absolute value. Consequentially, within the space $C$, the distance $||y_1(x)-y_2(x)||_0$ between the functions $y_1(x)$ and $y_2(x)$ does not exceed $\epsilon$ if the graph of $y_2(x)$ lies within a strip of width $2 \epsilon$ centered at $y_1(x)$.

\

\item The space $D_1$, or $D_1(a,b)$, consisting of all continuous functions $y(x)$ defined on an interval $a \leq x \leq b$ with continuous first derivatives. The operations of addition and multiplication are the same as $C$, but the norm is defined as
$$||y||_1 = \underset{a \leq x \leq b}{max} \ |y(x)| + \underset{a \leq x \leq b}{max} \ |y'(x)| .$$
Hence, two functions in  are regarded as close if both the functions and their derivatives are close together, that is, $||y-z||_1 < \epsilon$ implies that for all $x \in [a,b]$ that $|y(x)-z(x)| < \epsilon$ that $|y(x)-z(x)| < \epsilon$ and $|y'(x)-z'(x)| < \epsilon$.

\

\item The space $D_n$, or $D_n(a,b)$, consisting of all continuous functions $y(x)$ defined on an interval $a \leq x \leq b$ which have continuous derivatives up to $n$, where $n \in \NN$.\footnote[2]{Note that $D_n \subset D_{n-1} \subset \dots \subset D_1 \subset C$.} Addition and multiplication of elements of $D_n$ are defined just as previous cases, but the norm is defined as
$$||y||_n=\sum_{i=0}^n  \ \underset{a \leq x \leq b}{max} \ |y^{(i)}(x)|.$$
Thus, two functions in $D_n$ are said to be close together if the values of the functions and all their derivatives up to $n$ are close together.
\end{enumerate}

\section{Functionals}
Functionals are the main study of Calculus of Variations as well as our tool for obtaining the geodesic on the surface of a function. An interpretation of this is that a functional is a kind of function, where the independent variable itself is a function (or curve).
\begin{definition}
A functional is a correspondence which assigns a definite (real) number to each function (or curve) belonging to some class, that is, $J: D_n \rightarrow \RR$, where we chose a reasonable integer $n$ such that $J$ is continuous.
\end{definition}

\begin{definition}
The functional $J[y]$ is said to be continuous at the point $\hat{y} \in \Re $ if for any $\epsilon > 0$, there is a $\delta > 0$ such that
$$|J[y]-J[\hat{y}]| < \epsilon$$
provided that $||y-\hat{y}||<\delta$.
\end{definition}

As a general example, the expression
$$J[y]=\int_a^b F(x,y,y') \ dx$$
where $y(x)$ ranges over the set of all continuously differentiable functions on $[a,b]$, defines a functional.

The integrand is known as the Lagrangian, and we usually assume $L(x,u,p)$ to be a reasonably smooth function of all three of its (scalar) variables $x, \ u,$ and $p$ representing $\frac{du}{dx}=u'$. By choosing different Lagrangians we can generate different functions. Here are some examples, the last of which is the two-dimensional analogue of our three-dimensional geodesic problem.
\begin{example}
Let $L(x,u,p) = p^2$ and $y(x)$ be an arbitrary continuously differentiable function defined on $[a,b]$. We then have
$$\int_a^b L(x,u,p) \ dx = \int_a^b \left( \frac{dy}{dx} \right) ^2  dx.$$
\end{example}

\begin{example}
Let $L(x,u,p) = \sqrt{1+p^2} \ dx$. The corresponding functional is the length of the plane curve $y=y(x)$ from \textbf{a} to \textbf{b}
$$\int_a^b \sqrt{1+\left( \frac{dy}{dx} \right)^2} dx.$$
\end{example}

In order to establish the proper connection between functionals of Calculus of Variations and the functions treated in Analysis, we may proceed as follows:

\

Consider the aforementioned functional $J[y]=\int_a^b F(x,y,y') \ dx$ with $y(a) = A$ and $y(b) = B$. Using the points
$$a = x_0, \ x_1, \ \dots , \ x_i, \ \dots , \ x_n, \ x_{n+1} = b$$
we divide the interval $[a,b]$ into $n+1$ equal parts. We then replace the curve $y=y(x)$ by the polygonal line with vertices
$$(x_0, A), \ (x_1, y(x_1)), \ \dots , \ (x_n,y(x_n)), \ (x_{n+1}, B)$$
and the approximation of the functional $J[y]$ would be given by the sum
$$J(y_1, \ \dots , \ y_n) = \sum_{i=1}^{n+1} \ F(x_i, y_i,  \frac{y_i-y_{i-1}}{h}) \Delta x$$
where $y_i=y(x_i)$ and $\Delta x = x - x_{i-1}$. \\

The above sum is therefore a function of the $n$ variables $y_1, \ \dots , \ y_n$ and the exact value of the functional can be found by taking $\underset{n \rightarrow \infty}{\lim}$. In this sense, functionals can be regarded as ``functions of infinitely many variables.''

\

The reason the above analogy works so well is because functionals of the type $J[y]=\int_a^b F(x,y,y') \ dx$ have a \textit{``localization property''} consisting of the fact that if we divide the curve $y=y(x)$ into parts and calculate the value of the functional of each part, the sum of the values of the functional for seperate parts equals the value of the functional for the whole curve\footnote[3]{Only functionals with the localization property are usually considered in Calculus of Variations.}.
\begin{example}
The following functional does not uphold the localization property. Assuming the curve $y(x)$, where $a \leq x \leq b$, is made of some homogeneous material, then the $x$-coordinate of the center of mass is provided by the functional
$$\frac{\int_a^b x \sqrt{1+ \left( \frac{dy}{dx} \right)^2} dx}{ \int_a^b \sqrt{1+ \left( \frac{dy}{dx} \right)^2} dx }.$$
\end{example}

\section{Variation of a Functional}
The concept of the variation is analogous to the concept of the differential of a function of $n$-variables, that is, just as setting $\frac{dy}{dx}=0$ assists in finding the extrema of $y=y(x)$, as does setting ``$\delta J = 0$'' help locate the extrema of $J=J[y]$.  This concept of $\delta J$ will be discussed at the end of the section.

\

We begin by introducing continuous linear functionals, then state a handful of important lemmas the result from our definition of function spaces.
\begin{definition}
Given a normed linear space $\Re$, let each element $h \in \Re$ be assigned a number $\phi [h]$, i.e., let $\phi [h]$ be a functional defined on $\Re$. Then $\phi [h]$ is said to be a (continuous) \textit{linear functional} if
\begin{enumerate}
\item $\phi[ah] = a\phi[h]$;
\item $\phi[h_1 + h_2] = \phi[h_1] + \phi[h_2]$;
\item $\phi[h]$ is continuous for all $h \in \Re$.
\end{enumerate}
\end{definition}

\noindent We shall make use of the following lemmas.
\begin{lemma}
If $\alpha(x)$ is continuous in $[a,b]$, and if
$$\int_a^b \alpha(x)h(x) \ dx = 0$$
for every function $h(x) \in C(a,b)$ such that $h(a)=h(b)=0$, then, for all $x \in [a,b]$, $\alpha(x) = 0$.
\end{lemma}
\begin{proof}
Suppose the function $\alpha(x)$ is positive and nonzero in some point within $[a,b]$. Then $\alpha(x)$ is also positive in some interval $[x_1,x_2] \subseteq [a,b]$. If we set
$$h(x) = (x-x_1)(x_2-x)$$
for $x \in [x_1,x_2]$ and $h(x)=0$ otherwise, then $h(x)$ would clearly satisfy the conditions of the lemma. However, since
$$\int_a^b \alpha(x)h(x) \ dx = \int_{x_1}^{x_2} \alpha(x)(x-x_1)(x_2-x)(x) \ dx > 0$$
(except at $x_1$ and $x_2$), this is a contradiction.
\end{proof}

\begin{lemma}
If $\alpha(x)$ is continuous in $[a,b]$, and if
$$\int_a^b \alpha(x)h'(x) \ dx = 0$$
for every function $h(x) \in D_1(a,b)$ such that $h(a)=h(b)=0$, then, for all $x \in [a,b]$, $\alpha(x) = c$, where $c$ is a constant.
\end{lemma}
\begin{proof}
Let $c$ be the constant defined by the condition
$$\int_a^b [\alpha(x) - c] \ dx = 0,$$
and let
$$h(x) = \int_a^x [\zeta(x) - c] \ d\zeta,$$
so that $h(x)$ automatically belongs to $D_1(a,b)$ and satisfies the conditions $h(a)=h(b)=0$. Then on one hand,
$$\int_a^b [\alpha(x) - c]h'(x) \ dx = \int_a^b \alpha(x) h'(x) dx - c[h(b)-h(a)] = 0,$$
while on the other hand,
$$\int_a^b [\alpha(x) - c]h'(x) \ dx = \int_a^b [\alpha(x)-c]^2 \ dx.$$
It follows that $\alpha(x)=c$ for all $x \in [a,b]$.
\end{proof}

\begin{lemma}
If $\alpha(x)$ is continuous in $[a,b]$, and if
$$\int_a^b \alpha(x)h''(x) \ dx = 0$$
for every function $h(x) \in D_2(a,b)$ such that $h(a)=h(b)=0$ and $h'(a)=h'(b)=0$, then, for all $x \in [a,b]$, $\alpha(x) = c_0 + c_1x$, where $c_0$ and $c_1$ are constants.
\end{lemma}
\begin{proof}
Let $c_0$ and $c_1$ be defined by the conditions
$$\int_a^b [\alpha(x) - c_0-c_1x] \ dx = 0,$$
$$\int_a^b \ dx \ \int_a^x [\alpha(\zeta) - c_0-c_1\zeta] \ dx = 0,$$
and let
$$h(x) = \int_a^x \ d\zeta \ \int_a^\zeta [a(t) -c_0 - c_1] \ dt,$$
so that $h(x)$ automatically belongs to $D_2(a,b)$ and satisfies the conditions $h(a)=h(b)=0$ and $h'(a)=h'(b)=0$. We have, on one hand,
$$\int_a^b [\alpha(x) - c_0 - c_1 x]h''(x) \ dx$$
$$= \int_a^b \alpha(x)h''(x) \ dx - c_0[h'(b)-h'(a)]-c_1 \int_a^b xh''(x) \ dx$$
$$= -c_1 [bh'(b) - ah'(a)] - c_1[h(b)-h(a)] = 0,$$
while on the other hand,
$$\int_a^b [\alpha(x) - c_0 - c_1x]h''(x) \ dx = \int_a^b [\alpha(x) - c_0 - c_1 x]^2 \ dx = 0.$$
It follows that $\alpha(x) - c_0 - c_1x = 0$, i.e., $\alpha(x)=c_0 + c_1x$, for all $x$ in $[a,b]$.
\end{proof}

\begin{lemma}
If $\alpha(x)$ and $\beta(x)$ are continuous in $[a,b]$, and if
$$\int_a^b [ \alpha (x) h(x) + \beta (x) h'(x) ] \ dx = 0$$
for every function $h(x) \in D_1(a,b)$ such that $h(a)=h(b)=0$, then $\beta(x)$ is differentiable and $\beta '(x) = \alpha(x)$ for all $x \in  [a,b]$.
\end{lemma} 
\begin{proof}
Setting
$$A(x) = \int_a^x \alpha(\zeta) \ d\zeta,$$
and integrating by parts, we find that
$$\int_a^b \alpha(x)h(x) \ dx = - \int_a^b A(x)h'(x) \ dx,$$
i.e., the integral within the lemma may be rewritten as
$$\int_a^b [-A(x)+\beta(x)]h'(x) \ dx = 0.$$
According to Lemma 4.2, this implies that
$$B'(x)=\alpha(x),$$
for all $x$ in $[a,b]$, as asserted. Note that $\beta(x)$ was not assumed to be differentiable but results from Lemma 4.2.
\end{proof}

Let $J[y]$ be a functional defined on some normed linear space, and let
$$\Delta J[h] = J[y+h] -J[y]$$
 Be its increment , corresponding to the increment of $h=h(x)$ of the "independent variable" $y=y(x)$. If $y$ is fixed, $\Delta J[h]$ is a functional of $h$, that is, it is a nonlinear functional, similar to that of the term $f(x+h)-f(x)$ within  Newton's difference quotient.

Now suppose that
$$\Delta J[h] = \phi [h] + \epsilon ||h||,$$
where $\phi [h]$ is a linear functional and $\epsilon \rightarrow 0$ as $||h|| \rightarrow 0$ The functional $J[y]$ is said to be \textit{differentiable} and the \textit{principle linear part} of the increment $\Delta J[h]$ (i.e. the linear functional $\phi[h]$) and is called the \textit{variation} (or differential) of $J[y]$ and is denoted by $\delta J[y]$.\footnote[4]{Taking $\epsilon \rightarrow 0$ and $||h|| \rightarrow 0$ is similar to taking having $h \rightarrow 0$ in $\frac{f(x+h)-f(x)}{h}$.}

\begin{theorem}
A necessary condition for the differentiable function $J[y]$ to have an extremeum for $y=\hat{y}$  is that its variation to vanish for $y=\hat{y}$, i.e. that
$$\delta J[h] = 0$$
for $y=\hat{y}$   and all admissable $h$.
\end{theorem}

\noindent To analogize with Analysis, recall the following:

Let $F(x_1, \ \dots, \ x_n)$ be a differentiable function of $n$-variables. Then $F(x_1, \ \dots, \ x_n)$ is said to have a relative extremum at the point $(\hat{x_1}, \ \dots, \ \hat{x_n})$ if 
$$\Delta F = F(x_1, \ \dots, \ x_n) - F(\hat{x_1}, \ \dots, \ \hat{x_n})$$
has the same sign for all points $(x_1, \ \dots, \ x_n)$ belonging to some neighborhood of $(\hat{x_1}, \ \dots, \ \hat{x_n})$, where the extrema $ F(\hat{x_1}, \ \dots, \ \hat{x_n})$ is minimum if $\Delta F \geq 0$ and maximum if $\Delta F \leq 0$.

Similarly, we say the functional $J[y]$ has a (relative) extremum for $y=\hat{y}$ if $J[y] - J[\hat{y}]$ \textit{does not} change sign in some neighborhood of the curve $y=\hat{y}(x)$.

\section{The Euler-Lagrange Equation}
We now outline a method for finding the extremum of the simplest type of functional, developed by Euler and Lagrange. This method is our main tool for obtaining geodesics.

\noindent Lets assume we must face the following problem:

\

Let $F(x,y,z)$ be a function with continuous first and second (partial) derivatives with respect to all its arguments. Then among all functions $y(x)$ which are continuously differentiable for $a \leq x \leq b$ and satisfy the boundary conditions $y(a) = A$ and $y(b)=B$, find the function for which the functional
$$J[y] = \int_a^b F(x,y,y') dx$$
has a relative minimum.

\

Let us begin our solution. Suppose we give $y(x)$ an increment $h(x)$, where, in order for the function
$$y(x)+h(x)$$
to continue to satisfy the boundary conditions, we must have $h(a)=h(b)=0$. Then, since the corresponding increment of the functional equals
$$\Delta J = J[y+h] - J[y]$$
$$= \int_a^b F(x,y+h,y'+h') \ dx - \int_a^b F(x,y,y')  \  dx$$
$$= \int_a^b \left[ F(x,y+h,y'+h') - F(x,y,y') \right] \  dx$$
it follows from Taylor's Theorem that
$$\Delta J = \int_a^b \left[ h \frac{\partial}{\partial y} F(x,y,y') + h' \frac{\partial}{\partial y'} F(x,y,y') \right] \  dx + \dots$$
where the ``$\dots$'' denote terms of order higher than $1$ relative to $h$ and $h'$. The integrand on the right hand side represents the principle linear part of $\Delta J$, hence the variation of $J[y]$ is
$$\delta J = \int_a^b \left[ h \frac{\partial}{\partial y} F(x,y,y') + h' \frac{\partial}{\partial y'} F(x,y,y') \right] \  dx .$$
According to Theorem 4.5, a necessary condition for $J[y]$ to have an extremum at $y=y(x)$ is that
$$\delta J = \int_a^b \left[ h F_y+ h'  F_{y'} \right] \  dx = 0 $$
for all admissible $h$. But, according to Lemma 4.4, this implies that
$$F_y- \frac{d}{dx} F_{y'} = 0$$
a result known as the Euler-Lagrange Equation.

\

\noindent Thus, we have proved

\begin{theorem}
Let J[y] be a functional of the form
$$\int_a^b F(x,y,y') \ dx,$$
defined on the set of functions $y(x)\in D_1(a,b)$, i.e., which have continuous first derivatives in $[a,b]$, and satisfy the boundary conditions $y(a)=A$, $y(b)=B$. Then a necessary condition for $J[y]$ to have an extremum for a given function $y(x)$ is that $y(x)$ satisfy the Euler-Lagrange equation
$$F_y- \frac{d}{dx} F_{y'} = 0.$$
\end{theorem}

The curves satisfying the Euler-Lagrange equation are called \textit{extremals}. Since it is a second-order differential equation, its solution will in general depend on two arbitrary constants, which are determined from the boundary conditions $y(a)=A$ and $y(b)=B$.

\section{Geodesics}

Suppose we have a surface $\mathcal{S}$ defined by the vector equation
\begin{equation}\label{vec}
\vec{r}(u,v)=x(u,v)\hat{i}+y(u,v)\hat{j}+z(u,v)\hat{k}.
\end{equation}
The shortest curve lying on $\mathcal{S}$ and connecting two points of $\mathcal{S}$ is called the \textit{geodesic} connecting the two points. We are now well prepared to begin our discussion of obtaining this curve.

\

The geodesic curve lying on surface $\mathcal{S}$ can be specified by the equations
$$u=u(t) \ \ \ \ \ v=v(t)$$
and can be found by minimizing the arc length integral
$$L \ = \ \displaystyle{\int_{t_0}^{t_1}} \frac{ds}{dt} \ dt = \displaystyle{\int_{t_0}^{t_1}} \ \sqrt{\left( \frac{dx}{dt} \right)^2+{\left( \frac{dy}{dt} \right) }^2 + {\left( \frac{dz}{dt} \right)}^2} \ dt$$
But, since $x$, $y$, and $z$ are functions of more than one variable, chain rule states
$$\frac{dx}{dt}=\frac{\partial x}{\partial u}\frac{du}{dt} + \frac{\partial x}{\partial v}\frac{dv}{dt}$$
$$\left( \frac{dx}{dt} \right) ^2 = \left( \frac{\partial x}{\partial u} \right) ^2 \left( \frac{du}{dt} \right) ^2 + 2 \ \frac{\partial x}{\partial u} \frac{\partial x}{\partial v} \frac{du}{dt}\frac{dv}{dt}+\left( \frac{\partial x}{\partial v} \right) ^2 \left( \frac{dv}{dt} \right) ^2$$
and similarly for $\displaystyle{\left( \frac{dy}{dt} \right) ^2}$ and $\displaystyle{\left( \frac{dz}{dt} \right) ^2}$. This results in
\begin{eqnarray*}
L &=& \int_{t_0}^{t_1} \bigg\{ \bigg[ \left( \frac{\partial x}{\partial u} \right) ^2 +\left( \frac{\partial y}{\partial u} \right) ^2 + \left( \frac{\partial z}{\partial u} \right) ^2 \bigg] \left( \frac{du}{dt} \right) ^2 \\
\; \;  \;  \;  \; \;  \; \; \;  \;  &+& 2 \ \bigg[ \frac{\partial x}{\partial u} \frac{\partial x}{\partial v} + \frac{\partial y}{\partial u}  \frac{\partial y}{\partial v} + \frac{\partial z}{\partial u}  \frac{\partial z}{\partial v} \bigg]\frac{du}{dt} \frac{dv}{dt} \\
\; \;  \;  \;  \; \;  \; \; \;  \;  &+& \bigg[ \left( \frac{\partial x}{\partial v} \right) ^2 +\left( \frac{\partial y}{\partial v} \right) ^2 + \left( \frac{\partial z}{\partial v} \right) ^2 \bigg] \left( \frac{dv}{dt} \right) ^2 \bigg\}^{\frac{1}{2}} \ dt.
\end{eqnarray*} 
Which can be rewritten as
$$J[u,v]=\int_{t_0}^{t_1} \sqrt{Eu'^2+2Fu'v'+Gv'^2} \ dt,$$
where $E$, $F$, and $G$ are the coefficients of the \textit{first fundamental (quadradic) form of the surface}, i.e.,
$$\begin{cases} \displaystyle{E=\vec{r}_u\cdot \vec{r}_u = \left( \frac{\partial x}{\partial u} \right) ^2 +\left( \frac{\partial y}{\partial u} \right) ^2 + \left( \frac{\partial z}{\partial u} \right) ^2} \\  \\ \displaystyle{F=\vec{r}_u \cdot \vec{r}_v = \frac{\partial x}{\partial u} \frac{\partial x}{\partial v} + \frac{\partial y}{\partial u}  \frac{\partial y}{\partial v} + \frac{\partial z}{\partial u}  \frac{\partial z}{\partial v}}\\ \\  \displaystyle{G=\vec{r}_v\cdot \vec{r}_v=\left( \frac{\partial x}{\partial v} \right) ^2 +\left( \frac{\partial y}{\partial v} \right) ^2 + \left( \frac{\partial z}{\partial v} \right) ^2}. \end{cases}$$

The Euler-Lagrange equation in this case corresponds to the two different equations
$$\displaystyle{ F_u- \frac{d}{dt} F_{u'} = 0,} \ \ \ \ \ \ \ \displaystyle{F_v- \frac{d}{dt} F_{v'} = 0,}$$
hence, we obtain
\begin{equation}\label{geodesiceq1}
\frac{E_u u'^2+2F_u u' v' + G_u v'^2}{\sqrt{Eu'^2+2Fu'v'+Gv'^2}} - \frac{d}{dt} \frac{2(Eu'+Fv')}{\sqrt{Eu'^2+2Fu'v'+Gv'^2}} = 0,
\end{equation}
\begin{equation}\label{geodesiceq2}
\frac{E_v u'^2+2F_v u' v' + G_v v'^2}{\sqrt{Eu'^2+2Fu'v'+Gv'^2}} - \frac{d}{dt} \frac{2(Fu'+Gv')}{\sqrt{Eu'^2+2Fu'v'+Gv'^2}} = 0.
\end{equation}

Which are the two differential equations whose solutions provide the geodesic on surface $\mathcal{S}$. \\

The above case has no restrictions due to the parametrization $u=u(t)$ and $v=v(t)$. Let us quickly observe a different approach with a single restriction, although resulting in helpful special cases.

Let once more the surface $\mathcal{S}$ be given by the vector equation \ref{vec}. In terms of the differentials of $u$ and $v$, the square of the differential of arc length may be wrriten
\begin{eqnarray*}
(ds)^2&=&(dx)^2+(dy)^2+(dz)^2 \\
&=&Eu'^2+2Fu'v'+Gv'^2
\end{eqnarray*}
If the given fixed points on the surface are $(u_1,v_1)$ and $(u_2,v_2)$, with $u_2>u_1$, and we limit our consideration to arcs whos equations are expressible in the form
\begin{equation}
v=v(u) \ \ \ \ \ v(u_1) \leq v \leq v(u_2),
\end{equation}
the functional would be instead given by
\begin{equation}
J[v(u)] = \int_{u_1}^{u_2}  ds = \int_{u_1}^{u_2} \sqrt{E+2Fv'+Gv'^2} \ du
\end{equation}
where $v'=v'(u)$ designates the derivative $\frac{dv}{du}$ (hence $u'=\frac{du}{du}=1$). With $u$ and $v$ playing the roles of $x$ and $y$, the Euler-Lagrange equation now becomes
$$F_v- \frac{d}{du} F_{v'} = 0$$
\begin{equation}
\frac{E_v+2v'F_v+v'^2G_v}{2\sqrt{E+2Fv'+Gv'^2}} - \frac{d}{du} \left( \frac{F+Gv'}{\sqrt{E+2Fv'+Gv'^2}} \right) =0.
\end{equation}
Now we can observe some special cases resulting from this different approach.

\begin{remark}
In the case when $E$, $F$, and $G$ are explicit functions of $u$ only, we have
\begin{equation}
\frac{F+Gv'}{\sqrt{E+2Fv'+Gv'^2}} = c_1,
\end{equation}
which can be solved for $v'$. Doing this gives us
\begin{eqnarray*}
&&v'=\frac{1}{2G(G-{c_1}^2)} \bigg[ 2F({c_1}^2-G) \\
&& \pm \sqrt{ 4F^2(G-{c_1}^2)^2-4G(G-{c_1}^2)(F^2-E{c_1}^2) } \bigg].
\end{eqnarray*}
\end{remark}

\begin{remark}
In the case where $E$ and $G$ are explicit functions of $u$ only and $F=0$, that is, the grid lines of $u$ and $v$ are perpendicular, we have
$$v'=\frac{\sqrt{4G(G-{c_1}^2)E{c_1}^2}}{2G(G-{c_1}^2)}=c_1 \sqrt{\frac{E}{G(G-{c_1}^2)}},$$
so we get
\begin{equation}
v=c_1 \int_{u_1}^{u_2} c_1 \sqrt{\frac{E}{G(G-{c_1}^2)}} \ du.
\end{equation}
\end{remark}

\begin{remark}
Similarly, in the case where $E$ and $G$ are explicit functions of $v$ only and $F=0$, then
$$\frac{E_v+v'^2G_v}{2\sqrt{E+Gv'^2}}-\frac{d}{du} \left( \frac{Gv'}{\sqrt{E+Gv'^2}} \right) =0,$$
so
$$ E_v + v'^2G_v - 2 G \sqrt{E+Gv'^2} \bigg[ \frac{v''}{\sqrt{E+Gv'^2}} + \left( \frac{1}{2} \right) \frac{v'(2Gv'v'')}{(E+Gv'^2)^{\frac{3}{2}}} \bigg] =0 $$
$$E_v + v'^2 G_v - 2Gv'' + \frac{2G^2v'^2v''}{E+Gv'^2} = 0 $$
\begin{equation}
\frac{Gv'^2}{\sqrt{E+Gv'^2}} - \sqrt{E+Gv'^2} = c_1 
\end{equation}
which can be made even more helpful by noting $\displaystyle{v'= \frac{dv}{du}}$, giving us
$$Gv'^2-(E+Gv'^2)=c_1 \sqrt{E+Gv'^2}$$
$$\left( - \frac{E}{c_1} \right) ^2 = E + Gv'^2$$
$$\frac{E^2-{c_1}^2E}{G{c_1}^2} = v'^2,$$
finally providing
\begin{equation}
u={c_1} \int_{v_1=v(u_1)}^{v_2=v(u_2)} \sqrt{ \frac{G}{E^2-{c_1}^2E }} \ dv
\end{equation}
\end{remark}

\begin{example}
For a surface obtained through revolving $y=f(x)$ about the $x$-axis over the interval $a \leq x \leq b$ parametrized by
$$x=u \ \ \ \ \ y=f(u) \ cos(v) \ \ \ \ \ z=f(u) \ sin(v)$$
where $a \leq u \leq b$ and $0 \leq v \leq 2 \pi$, the first fundamental form would be
$$E=1+(f'(u))^2 \ \ \ \ \ G=(f(u))^2 \ \ \ \ \ F=0,$$
and by applying the result of Remark 6.2, the geodesic would be given by
\begin{equation}
v=c_1 \int_{u_1}^{u_2} \frac{\sqrt{1+(f'(u))^2}}{f(u)\sqrt{(f(u))^2 - {c_1}^2}} \ du.
\end{equation}
\end{example}

\

\begin{example}
Consider the circular cylinder
$$\vec{r}(\varphi, z)=(a \ cos\varphi)\hat{i} + (a \ sin\varphi)\hat{j} + z\hat{k}.$$

\noindent Note the coeffecients of the first fundamental form of the cylinder are
$$E=a^2 \ \ \ \ \ F=0 \ \ \ \ \ G=1,$$
therefore the geodesics of the cylinder have the equations, according to the first method,
$$\frac{d}{dt} \frac{a^2 \varphi'}{\sqrt{a^2\varphi'^2+z'^2}} = 0, \ \ \ \ \ \ \ \frac{d}{dt} \frac{z'}{\sqrt{a^2\varphi'^2+z'^2}} = 0,$$
i.e.,
$$\frac{a^2 \varphi'}{\sqrt{a^2\varphi'^2+z'^2}} = C_1, \ \ \ \ \ \ \ \frac{z'}{\sqrt{a^2\varphi'^2+z'^2}} = C_2.$$
Dividing the second equation by the first, we obtain
$$\frac{dz}{d\varphi} = c_1,$$
which has the solution
$$z=c_1 \varphi + c_2,$$
representing a two-parameter family of helical lines lying on the cylinder.
\end{example}

\begin{example}
Now, consider a sphere of radius $a$ given by the vector equation
\begin{equation}\label{sphere}
\vec{r}(u,v)=(a \ sin(v) \ cos(u))\hat{i} + (a \ sin(v) \ sin(u))\hat{j} + (a \ cos(v))\hat{k}
\end{equation}
where it should be simple to notice that $u$ and $v$ form orthogonal grid lines, due to them representing the latitude and longitude, respectively. The first fundamental form of this sphere is
$$E=a^2 \ sin^2(v) \ \ \ \ \ G=a^2 \ \ \ \ \ F=0.$$
Using the result of Remark 6.3, we have
\begin{eqnarray*}
u &=& c_1 \int \frac{dv}{\sqrt{a^2 \ sin^4(v) - {c_1}^2 \ sin^2(v)}} \\
&=& \int \frac{csc^2 (v) \ dv}{\sqrt{ \left( \left( a/{c_1} \right)^2 - 1 \right) - cot^2(v) }} \\
&=& -sin^{-1} \left( \frac{cot(v)}{\sqrt{(a/c_1)^2 -1}} \right) + c_2
\end{eqnarray*}
when rearranged provides
$$(sin(c_2))a \ sin(v) \ cos(u) - (cos(c_2))a \ sin(v) \ sin(u) - \frac{a \ cos(v)}{\sqrt{(a/c_1)^2 - 1}} = 0.$$
Lastly, upon which noting the parametric equations in equation \ref{sphere}, we may change the above to see that the geodesic of a sphere lies on
$$x \ sin(c_2) - y \ cos(c_2) - \frac{z}{\sqrt{(a/c_1)^2-1}}=0,$$
which is a plane passing through the center of the sphere. Hence, the shortest arc connecting two points on the surface of a sphere is the intersection of the sphere with the plane containing the two points and the center of the sphere, known as the \textit{great-circle arc}.
\end{example}

The concept of a geodesic can be defined not only for surface, by also for higher-dimensional manifolds. Find the geodesics of an $n$-dimensional manifold reduces to solving a variational problem for a functional depending on $n$ functions.

\nocite{*}

\printbibliography

\end{document}